\documentclass[12pt,a4paper,oneside,reqno]{amsart}

\usepackage{fouriernc}

\usepackage{cite,url,hyperref}

\newcommand{\uhr}{\mathbin{\lceil}}

\newcommand{\dd}{\mathrm{d}}
\renewcommand{\Tilde}{\widetilde}

\renewcommand{\Bar}{\overline}

\newcommand{\RR}{\mathbb{R}}
\newcommand{\CC}{\mathbb{C}}
\newcommand{\NN}{\mathbb{N}}

\newcommand{\mx}{\mathrm{max}}

\newtheorem{theorem}{Theorem}

\newtheorem{lemma}[theorem]{Lemma}

\theoremstyle{definition}

\newtheorem{remark}[theorem]{Remark}

\DeclareMathOperator{\spec}{spec}
\DeclareMathOperator{\dom}{dom}

\begin{document}

\title[]{On the asymptotics of the principal eigenvalue\\
for a Robin problem with a large parameter\\ in planar domains}

\author{Konstantin Pankrashkin}

\address{Laboratoire de math\'ematiques -- UMR 8628, Universit\'e Paris-Sud, B\^atiment 425, 91405 Orsay Cedex, France}

\thanks{Written for the proceedings
of the conference ``Mathematical Challenge of Quantum Transport in Nanosystems'' (Pierre Duclos Workshop), Saint-Petersburg, Russia, March 13--15, 2013}

\email{konstantin.pankrashkin@math.u-psud.fr}
\urladdr{http://www.math.u-psud.fr/~pankrash/}

\begin{abstract}
Let $\Omega\subset \RR^2$ be a domain having a compact boundary $\Sigma$ which is
Lipschitz and piecewise $C^4$ smooth,
and let $\nu$ denote the inward unit normal vector on~$\Sigma$.
We study the  principal eigenvalue $E(\beta)$ of the Laplacian in $\Omega$ with the Robin boundary conditions
$\partial f/\partial\nu +\beta f=0$ on $\Sigma$, where $\beta$ is a positive number. 
Assuming that $\Sigma$ has no convex corners we show the estimate $E(\beta)=-\beta^2- \gamma_\mx\beta + O\big(\beta^\frac{2}{3}\big)$ as $\beta\to+\infty$,
where $\gamma_\mx$ is the maximal curvature of the boundary.

\medskip


\medskip

\end{abstract}

\maketitle

\section{Introduction}

Let $\Omega\subset \RR^2$ be an open connected set having a compact Lipschitz piecewise smooth boundary $\Sigma$.
For $\beta>0$ consider the operator  $H_\beta$ which is the Laplacian $f\mapsto-\Delta f$
with the Robin boundary conditions,
\[
\dfrac{\partial f}{\partial \nu}+\beta f=0 \text{  on }\Sigma,
\]
where $\nu$ is the inward unit normal vector. More precisely, $H_\beta$ 
is the self-adjoint operator in $L^2(\Omega)$ associated with the sesquilinear form
\begin{equation}
     \label{eq-fhb}
h_\beta (f,g)=\iint_\Omega \Bar{\nabla f}\nabla g\, \dd x -\beta \int_\Sigma \Bar f\, g\,\dd\sigma,
\quad \dom h_\beta=H^1(\Omega);
\end{equation}
here $\sigma$ denotes the one-dimensional Hausdorff measure on $\Sigma$.
The operator $H_\beta$ is semibounded from below. If $\Omega$ is bounded, then
$H_\beta$ has a compact resolvent, and we denote by $E_j(\beta)$, $j\in\NN$,
its eigenvalues taken according to their multiplicities
and enumerated in the non-decreasing order.
If $\Omega$ is unbounded, then the essential spectrum of $H_\beta$
coincides with $[0,+\infty)$, and the discrete spectrum consists of finitely many 
eigenvalues which we denote again by $E_j(\beta)$, $j\in\{1,\dots,N_\beta\}$,
and enumerate them in the non-decreasing order taking into account the multiplicities.
In the both cases the principal eigenvalue $E(\beta):=E_1(\beta)$ may be defined through the Rayleigh quotients
\[
E(\beta)=
\inf_{0\ne f\in \dom h_\beta}\dfrac{h_\beta(f,f)}{\|f\|^2_{L^2(\Omega)}}.
\]
It is easy to check that $E(\beta)<0$: for bounded $\Omega$ one can test on $f=1$,
and for unbounded $\Omega$ one may use $f(x)=\exp\big(-|x|^\alpha/2\big)$ with small $\alpha>0$.
 
The study of the principal eigenvalue arises in several applications:
the work \cite{bol} discusses the stochastic meaning of the Robin eigenvalues,
the paper~\cite{lacey} shows the role of the eigenvalue problem appears in the study
of a long-time dynamics related to some reaction-diffusion process, and
a discussion of an interplay between the eigenvalues and the estimate of the critical temperature in a problem
of superconductivity may be found in \cite{gs1}.

In the present note we are interested in the asymptotic behavior of $E(\beta)$ for large values of $\beta$.
For bounded $\Omega$, this question was already addressed in numerous papers.
It was conjectured and partially proved in \cite{lacey} that one has the asymptotics
\begin{equation}
        \label{eq-com}
E(\beta)=-C_\Omega \beta^2+o(\beta^2) \text{ as } \beta\to+\infty
\end{equation}
for some constant $C_\Omega>0$. It seems that the paper \cite{luzhu}
contains the first rigorous proof of the above equality for the case  of a $ C^1$ smooth $\Sigma$,
and in that case one has $C_\Omega=1$, as predicted in~\cite{lacey}. Under the
same assumption, it was shown in \cite{dan1} that the same asymptotics
$E_j(\beta)=-\beta^2 +o(\beta^2)$, $\beta\to+\infty$,
holds for any fixed $j\in\NN$. The paper \cite{lp} proved the asymptotics \eqref{eq-com} for
domains whose boundary is $C^\infty$ smooth
with a possible exception of finitely many corners.
If the corner opening angles are $\alpha_j\in(0,\pi)\mathop{\cup}(\pi,2\pi)$, $j=1,\dots, m$,
and $\theta:=\min \alpha_j/2$, then
$C_\Omega=(\sin \theta)^{-2}$ if $\theta <\pi/2$, otherwise $C_\Omega=1$.
We remark that the paper \cite{lp} formally deals with bounded domains, but
the proofs  can be easily adapted to unbounded domains with compact boundaries.
It should pointed out that domains with cusps need a specific consideration, and the results
are different~\cite{kov2,lp}. Various generalizations
of the above results and some related questions concerning the spectral theory
of the Robin Laplacians were discussed e.g. in the papers~\cite{colgar,dan2,frg,gs2,kov2,KL}.
The aim of the present note is to refine the asymptotics \eqref{eq-com} for a class of
two-dimensional domains. More precisely, we calculate the next term in the asymptotic
expansion for piecewise $C^4$ smooth domains whose boundary has no convex corners, i.e. we assume that
either the boundary is smooth or that all corner opening angles are larger than $\pi$;
due to the above cited result of \cite{lp} we have $C_\Omega=1$ in the both cases.

Let us formulate the assumptions and the result more carefully.
Let $\Sigma_k$, $k=1,\dots,n$, be non-intersecting $C^4$ smooth connected components of the boundary $\Sigma$
such that $\Sigma=\bigcup_{k=1}^n\Bar\Sigma_k$. Denote by $\ell_k$ the length of $\Sigma_k$ and consider
a parametrization of the closure $\Bar\Sigma_k$ by the arc length, i.e.
let $[0,\ell_k]\ni s\mapsto \Gamma_k(s)\equiv\big(\Gamma_{k,1}(s),\Gamma_{k,2}(s)\big)\in \Bar\Sigma_k$ be a bijection
with $|\Gamma'_k|=1$ and such that $\Gamma_k\in C^4\big([0,\ell_k]\big)$, and
we assume that the orientation of each $\Gamma_k$ is chosen in such a way
that $\nu_k(s):=\big(-\Gamma'_{k,2}(s),\Gamma'_{k,1}(s)\big)$ is the inward unit normal vector at the point $\Gamma_k(s)$ of the boundary.
If two components $\Sigma_j$, $\Sigma_k$ meet at some point $P:=\Gamma_j(\ell_j)=\Gamma_k(0)$,
then two options are allowed: either $\Bar{\Sigma_j\cup\Sigma_k}$ is $C^4$ smooth near $P$
or the corner opening angle at $P$ measured inside $\Omega$ belongs to $(\pi,2\pi)$.

Denote by $\gamma_k(s)$ the signed curvature of the boundary
at the point $\Gamma_k(s)$ and let $\gamma_\mx$ denote its global maximum: 
\[
\gamma_k(s):=\Gamma'_{k,1}(s)\Gamma''_{k,2}(s)-\Gamma''_{k,1}(s)\Gamma'_{k,2}(s),\quad
\gamma_\mx:=\max_{k\in\{1,\dots,n\}} \max_{s\in[0,\ell_k]} \gamma_k(s);
\]
note that the decomposition of the boundary $\Sigma$ into the pieces $\Sigma_k$ is non-unique, but
the value $\gamma_\mx$ is uniquely determined. 
Our result is as follows:
\begin{theorem}\label{thm1}
Under the preceding assumptions there holds
\[
E(\beta)=-\beta^2-\gamma_\mx \beta +O\big(\beta^\frac{2}{3}\big) \text{ as } \beta\to+\infty.
\]
\end{theorem}

We believe that it is hard to improve the asymptotics without any additional
information on the set at which the curvature attains its maximal value.
For example, one may expect that the case of a curvature having isolated maxima and the case of a piecewise
constant curvature should give different resolutions of the remainder,
and we hope to progress in this direction in subsequent works.

At the first sight, the Robin eigenvalue problem may look rather similar to the eigenvalue problem
for $\delta$-potentials supported by curves, see e.g.~\cite{EY,EP,lot}.
This first impression is wrong, and the result of Theorem~\ref{thm1} concerning the secondary asymptotic term
is very different from the one obtained in the papers~\cite{EP,EY} for strong $\delta$-potentials;
nevertheless, a part of the machinery of \cite{EY} plays an important role in our considerations.
On the other hand, the asymptotic behavior of the principal Robin eigenvalue
shows some analogy with the lowest eigenvalue of the Neumann magnetic
laplacian studied in the theory of superconductivity~\cite{BN,FH,HM}.

\section{Dirichlet-Neumann bracketing on thin strips}\label{sec2}

In this section we introduce and study an auxiliary eigenvalue problem, and the result obtained will be used
in the next section to prove theorem~\ref{thm1}.

Let $\ell>0$ and let $\Gamma:[0,\ell]\to \RR^4$, $s\mapsto\Gamma(s)=\big(\Gamma_1(s),\Gamma_2(s)\big)\in\RR^2$,
be an injective $C^4$ map such that $\big|\Gamma'(s)\big|=1$ for all $s\in(0,\ell)$.
Denote
\begin{gather*}
S:=\Gamma\big((0,\ell)\big), \quad
\kappa(s):=\Gamma'_1(s)\Gamma''_2(s)-\Gamma''_1(s)\Gamma'_2(s),\quad
\kappa_\mx:=\max_{s\in[0,\ell]}\kappa(s),\\
K:=\max_{s\in[0,\ell]}\big|\kappa(s)\big|+
\max_{s\in[0,\ell]}\big|\kappa'(s)\big|+
\max_{s\in[0,\ell]}\big|\kappa''(s)\big|.
\end{gather*}
Due to $\kappa\in C^2\big([0,\ell]\big)$ the above quantity $K$ is finite.

For $a>0$ consider the map
\[
\Phi_a:(0,\ell)\times\RR\to \RR^2,\quad
\Phi_a(s,u)=\begin{pmatrix}
\Gamma_1(s)-u\Gamma'_2(s)\\
\Gamma_2(s)+u\Gamma'_1(s)
\end{pmatrix}.
\]
As shown in \cite[Lemma 2.1]{EY}, for any $a\in(0,a_0)$, $a_0:=(2K)^{-1}$,
the map $\Phi_a$ defines a diffeomorphism between the domains $\square_a:=(0,\ell)\times (0,a)$
and $\Omega_a:=\Phi_a(\square_a)$. In what follows we always assume
that $a\in(0,a_0)$ and we will work with the usual Sobolev space $H^1(\Omega_a)$
and its part
$\Tilde H^1_0(\Omega_a):=\big\{
f\in H^1(\Omega_a): \, f\uhr_{\partial \Omega_a\setminus\Bar S}=0
\big\}$;
here the symbol $\uhr$ means the trace of the function
on the indicated part of the boundary.

Introduce two sesquilinear forms in $L^2(\Omega_a)$. The first one $h^{N,a}_\beta$,
is defined on $\dom h^{N,a}_\beta:=H^1(\Omega_a)$ by the expression
\[
h^{N,a}_\beta(f,g)=\iint_{\Omega_a} \Bar{\nabla f} \nabla g \,\dd x-\beta\int_S \Bar f g\, \dd\sigma,
\] 
and the second one, $h^{D,a}_\beta$, it its restriction to 
$\dom h^{D,a}_\beta:=\Tilde H^1_0(\Omega_a)$. Both forms are densely defined, symmetric, closed and semibounded from below,
and we denote
\begin{equation}
      \label{eq-eig1}
E_{N/D}(\beta,a)=\inf_{0\ne f\in \dom h^{N/D,a}_\beta}\dfrac{h^{N/D,a}_\beta(f,f)}{\|f\|^2_{L^2(\Omega_a)}}.
\end{equation}
We are going to show the following result:
\begin{lemma}\label{prop1}
There exists $a_1>0$ such that for any $a\in(0,a_1)$
one has the estimate
$E_{N/D}(\beta,a)=-\beta^2-\kappa_\mx \beta +O(\beta^{\frac{2}{3}})$
as $\beta\to+\infty$.
\end{lemma}
The rest of this section is devoted to the proof of lemma~\ref{prop1}.
We first introduce a suitable decomposition of $\Omega_a$ and then use
an asymptotic separation of variables.

Define $U_a:L^2(\Omega_a)\to L^2(\square_a)$ by $\big(U_a f\big)(s,u)=\sqrt{1-u\kappa(s)}f\big(\Phi_a(s,u)\big)$.
Clearly, $U_a$ is a unitary operator, and one has $U_a\big(H^1(\Omega_a)\big)=H^1(\square_a)$ and
\[
U_a\big(\Tilde H^1_0(\Omega_a)\big)=\Tilde H^1_0(\square_a):=
\big\{
f\in H^1(\square_a):\, f(0,\cdot)=f(\ell,\cdot)=0 \text{ and } f(\cdot, a)=0
\big\},
\]
where the restrictions should be again understood as the traces.
Using the integration by parts one may easily check that
for any $f,g\in H^1(\Omega_a)$ one has
$h^{N,a}_\beta(f,g)=q^{N,a}_\beta(U_a f,U_a g)$,
where the form $q^{N,a}_\beta$ is defined on the domain $\dom q^{N,a}_\beta:=H^1(\square_a)$
by the expression
\begin{equation}
    \label{eq-qqn}
\begin{aligned}
q^{N,a}_\beta(f,g)=&\iint_{\square_a} \dfrac{1}{\big(1-u\kappa(s)\big)^2}\, \Bar{\dfrac{\partial f}{\partial s}}\dfrac{\partial g}{\partial s}\, \dd s\,\dd u
+\iint_{\square_a} \Bar{\dfrac{\partial f}{\partial u}}\dfrac{\partial g}{\partial u}\,\dd s\,\dd u\\
&-\iint_{\square_a} V(s,u)\Bar f \,g\,\dd s\,\dd u -\beta\int_0^{\ell} \Bar{f(s,0)}g(s,0)\, \dd s\\
&-\dfrac{1}{2}\int_0^{\ell} \kappa(s)\, \Bar{f(s,0)}g(s,0)\,\dd s
 +\dfrac{1}{2}\int_0^{\ell} \dfrac{\kappa(s)}{1-a\kappa(s)}\, \Bar{f(s,a)}g(s,a)\,\dd s\\
&+\dfrac{1}{2}\kappa'(\ell)\int_0^a \dfrac{u}{\big(1-u\kappa(\ell)\big)^3}\, \Bar{f(\ell,u)}g(\ell, u)\, \dd u\\
&-\dfrac{1}{2}\kappa'(0)\int_0^a \dfrac{u}{\big(1-u\kappa(0)\big)^3}\, \Bar{f(0,u)}g(0, u)\, \dd u
\end{aligned}
\end{equation}
with
\[
V(s,u):=\dfrac{u\kappa''(s)}{2\big(1-u\kappa(s)\big)^3}+\dfrac{5u^2\kappa'(s)^2}{4\big(1-u\kappa(s)\big)^4}
+\dfrac{\kappa(s)^2}{4\big(1-u\kappa(s)\big)^2}.
\]
Similarly, for any $f,g\in \Tilde H^1_0(\Omega_a)$ one has 
$h^{D,a}_\beta(f,g)=q^{D,a}_\beta(U_a f,U_a g)$,
where $q^{D,a}_\beta$ is the restriction of $q^{N,a}_\beta$ to
the domain $\dom q^{D,a}_\beta:=\Tilde H^1_0(\square_a)$;
note that for $f,g \in \dom q^{D,a}_\beta$
the three last terms on the right-hand side of \eqref{eq-qqn} vanish.
Using the unitarity of $U_a$ we may rewrite the equalities \eqref{eq-eig1} in the form
\begin{equation}
    \label{eq-eig2}
E_{N/D}(\beta,a)=\inf_{0\ne f\in \dom q^{N/D,a}_\beta}\dfrac{q^{N/D,a}_\beta(f,f)}{\|f\|^2_{L^2(\square_a)}}.
\end{equation}
We would like to reduce the estimation of these quantities to the  study
of the eigenvalues of certain one-dimensional operators.

Using the one-dimensional Sobolev inequality on $(0,\ell)$ we see that one can find a constant $C>0$ independent of $a$ such that for all $f\in H^1(\square_a)$ one has
\[
\int_0^a \big|f(0, u)\big|^2\, \dd u+\int_0^a \big|f(\ell, u)\big|^2\, \dd u
\le C \Big(\iint_{\square_a}\Big|\dfrac{\partial f}{\partial s}\Big|^2 \dd s\, \dd u+\iint_{\square_a}|f|^2 \dd s\, \dd u\Big)
\]
One can also find a constant $v>0$ such that
$\big|V(s,u)\big|\le v$ for all $(s,u)\in\square_a$ and all $a\in (0,a_0)$.
Furthermore, again for $(s,u)\in\square_a$ and any $a\in (0,a_0)$, we have
\[
\Big|\dfrac{\kappa(s)}{1-a\kappa(s)}\Big|\le 2K,
\quad
\dfrac{2}{3}\le \dfrac{1}{1-u\kappa(s)}\le 2.
\]
For any $M\in\NN$ we denote
\begin{gather*}
\delta:=\dfrac{\ell}{M},\quad
I_M^j:=( j\delta-\delta, j\delta),\quad
\square_{a,M}^j:=I_M^j\times(0,a),\\
\kappa^-_{M,j}:=\inf_{s\in I^j_M}\kappa(s), \quad
\kappa^+_{M,j}:=\sup_{s\in I^j_M}\kappa(s), \quad j=1,\dots,M,
\end{gather*}
and introduce functions $\kappa^\pm_M:(0,\ell)\to \RR$ as follows:
$\kappa^\pm_M(s):=\kappa^\pm_{M,j}$ if  $s\in I^j_M$, and $\kappa^\pm_M(j\delta):=0$ for $j=1,\dots,M-1$.
In addition, assume that $0<a<(10KC)^{-1}$. Now introduce two new sesquilinear forms
which will be used to obtain a two-side estimate for $E_{N/D}(\beta,a)$.
The first one, $t^{-,M,a}_\beta$, is defined by
\begin{align*}
\dom t^{-,M,a}_\beta&=H^1\big(\bigcup_{j=1}^M \square_{a,M}^j\big)\simeq
\bigoplus_{j=1}^M H^1\big(\square_{a,M}^j\big),\\
t^{-,M,a}_\beta(f,g)&=\Big(\dfrac{4}{9}-4aKC)
\iint_{\square_a} \Bar{\dfrac{\partial f}{\partial s}}\dfrac{\partial g}{\partial s} \,\dd s\,\dd u
+
\iint_{\square_a} \Bar{\dfrac{\partial f}{\partial u}}\dfrac{\partial g}{\partial u} \,\dd s\,\dd u\\
&\ - (v+4aKC)\iint_{\square_a} \Bar{f}g\, \dd s\,\dd u
-\int_0^{\ell}\Big(\beta+\dfrac{\kappa^+_M(s)}{2}\Big)\, \Bar{f(s,0)}g(s,0)\,\dd s\\
&\ -K\int_0^{\ell} \Bar{f(s,a)}g(s,a)\,\dd s.
\end{align*}
The second one, $t^{+,M,a}_\beta$, is defined on the domain
$\dom t^{+,M,a}_\beta=\bigoplus_{j=1}^M \Tilde H^1_0(\square_{a,M}^j)$,
\[
\Tilde H^1_0(\square_{a,M}^j):=
\big\{
f\in H^1(\square_{a,M}^j):\, f(j\delta-\delta,\cdot)=f(j\delta,\cdot)=0 \text{ and } f(\cdot, a)=0
\big\},
\]
through
\begin{multline*}
t^{+,M,a}_\beta(f,g)=4\iint_{\square_a} \Bar{\dfrac{\partial f}{\partial s}}\dfrac{\partial g}{\partial s} \,\dd s\,\dd u
+
\iint_{\square_a} \Bar{\dfrac{\partial f}{\partial u}}\dfrac{\partial g}{\partial u} \,\dd s\,\dd u\\
+ v\iint_{\square_a} \Bar{f}g\, \dd s\,\dd u -\int_0^{\ell}\Big(\beta+\dfrac{\kappa^-_M(s)}{2}\Big)\, \Bar{f(s,0)}g(s,0)\,\dd s.
\end{multline*}
One has clearly the inclusions $\dom t^{+,M,a}_\beta\subset \dom q^{D,a}_\beta\subset \dom q^{N,a}_\beta \subset \dom t^{-,M,a}_\beta$
and the inequalities
\begin{align*}
t^{-,M,a}_\beta(f,f)&\le q^{N,a}_\beta(f,f), & f&\in \dom q^{N,a}_\beta,\\
q^{N,a}_\beta(f,f)&=q^{D,a}_\beta(f,f), & f&\in \dom q^{D,a}_\beta,\\
q^{D,a}_\beta(f,f)&\le t^{+,M,a}_\beta(f,f),&  f&\in \dom t^{+,M,a}_\beta,
\end{align*}
which justify the estimates
\begin{equation}
    \label{eq-eig4}
E^-_M(\beta,a)\le E_N(\beta,a)\le E_D(\beta,a)\le E^+_M(\beta,a),
\end{equation}
where we denote
\[
E^\pm_M(\beta,a):=\inf_{0\ne f\in \dom t^{\pm,M,a}_\beta}\dfrac{t^{\pm,M,a}_\beta(f,f)}{\|f\|^2_{L^2(\square_a)}}
\]
Now we are going to estimate $E^\pm_M(\beta,a)$ using the separation of variables.
Note that the forms $t^{\pm,M,a}_\beta$ are densely defined, semibounded from below and closed in $L^2(\square_a)$,
therefore, they define some self-adjoint operators $T^{\pm,M,a}_\beta$ in $L^2(\square_a)$,
and $E^\pm_M(\beta,a)=\inf\,\spec T^{\pm,M,a}_\beta$.
On the other hand, due to the fact that the domains $\square_{a,M}^j$ are disjoint
and isometric to each other,
we can identify $T^{\pm,M,a}_\beta\simeq\bigoplus_{j=1}^M T^{\pm,M,a}_{\beta,j}$, 
where $T^{\pm,M,a}_{\beta,j}$ are self-adjoint operators acting in $L^2(\square_{\delta,a})$,
$\square_{\delta,a}:=(0,\delta)\times(0,a)$,
and associated respectively with the sesqulinear forms $t^{\pm,M,a}_{\beta,j}$,
\begin{align*}
t^{-,M,a}_{\beta,j}(f,g)&=\Big(\dfrac{4}{9}-4aKC)
\int_0^\delta\int_0^a \Bar{\dfrac{\partial f}{\partial s}}\dfrac{\partial g}{\partial s} \dd u\,\dd s
+
\int_0^\delta\int_0^a \Bar{\dfrac{\partial f}{\partial u}}\dfrac{\partial g}{\partial u} \dd u\,\dd s\\
&\ - (v+4aKC)\int_0^\delta\int_0^a \Bar{f}g\, \dd u\,\dd s
-\Big(\beta+\dfrac{\kappa^+_{M,j}}{2}\Big)\int_0^{\delta}\, \Bar{f(s,0)}g(s,0)\,\dd s\\
&\ -K\int_0^{\delta} \Bar{f(s,a)}g(s,a)\,\dd s,
\quad
\dom t^{-,M,a}_{\beta,j}=H^1(\square_{\delta,a}),\\
t^{+,M,a}_{\beta,j}(f,g)&=4\int_0^\delta\int_0^a \Bar{\dfrac{\partial f}{\partial s}}\dfrac{\partial g}{\partial s} \dd u\,\dd s
+
\int_0^\delta\int_0^a \Bar{\dfrac{\partial f}{\partial u}}\dfrac{\partial g}{\partial u} \dd u\,\dd s\\
&\ + v\int_0^\delta\int_0^a \Bar{f}g\, \dd u\,\dd s -\Big(\beta+\dfrac{\kappa^-_{M,j}}{2}\Big)\,\int_0^{\delta} \Bar{f(s,0)}g(s,0)\,\dd s,\\
&\dom t^{+,M,a}_{\beta,j}=\big\{
f\in H^1(\square_{\delta,a}):\, f(0,\cdot)=f(\delta,\cdot)=0 \text{ and } f(\cdot, a)=0
\big\}
\end{align*}
It is a routine to check that $T^{\pm,M,a}_{\beta,j}=Q^\pm_M\otimes 1 + 1\otimes L^{\pm,j}_{\beta,a}$,
where $Q^\pm_M$ are the operators acting in $L^2(0,\delta)$ as follows:
\begin{align*}
Q^-_M f&=-\big(\dfrac{4}{9}-4aKC\big)f''-(v+4aKC)f,\\
\dom Q^-_M&=\Big\{
f\in H^2(0,\delta): f'(0)=f'(\delta)=0
\Big\},\\
Q^+_M f&=-4f''+vf,\\
\dom Q^-_M&=\Big\{
f\in H^2(0,\delta): f(0)=f(\delta)=0
\Big\},
\end{align*}
and $L^{\pm,j}_{\beta,a}$ are the self-adjoint operators in $L^2(0,a)$ both acting as $L^{\pm,j}_{\beta,a} f=-f''$
on the domains
\begin{align*}
\dom L^{-,j}_{\beta,a}&=\Big\{
f\in H^2(0,a):\, f'(0)+\Big(\beta+\dfrac{\kappa^+_{M,j}}{2}\Big)f(0)=0,
\,
f'(a)-K f(a)=0
\Big\},\\
\dom L^{+,j}_{\beta,a}&=\Big\{
f\in H^2(0,a):\, f'(0)+\Big(\beta+\dfrac{\kappa^-_{M,j}}{2}\Big)f(0)=0,
\,
f(a)=0
\Big\}.
\end{align*}
The spectra of $Q^\pm_M$ can be calculated explicitly; in particular, one has
\[
\inf\,\spec Q^-_M=-v-4aKC, \quad
\inf\, \spec Q^+_M=\dfrac{4\pi^2}{\delta^2}+v\equiv\dfrac{4\pi^2 M^2}{\ell^2}+v.
\]
Therefore, denoting $E^{\pm,j}(\beta,a):=\inf\,\spec L^{\pm,j}_{\beta,a}$, we arrive at
\begin{equation}
 \label{eq-eig5}
 \begin{aligned}
E^-_M(\beta,a)&=\min_j \big(\inf\,\spec T^{-,M,a}_{\beta,j}\big)=-v-4aKC + \min_j E^{-,j}(\beta,a), \\
E^+_M(\beta,a)&=\min_j \big(\inf\,\spec T^{+,M,a}_{\beta,j}\big)=\dfrac{4\pi^2 M^2}{\ell^2}+v + \min_j E^{+,j}(\beta,a)
\end{aligned}
\end{equation}

To study the lowest eigenvalues of $L^{\pm,j}_{\beta,a}$ we prove two
auxiliary estimates.
 
\begin{lemma}\label{lem2}
For $a,\beta,\gamma>0$, let $\Lambda_{a,\beta,\gamma}$
denote the self-adjoint operator in $L^2(0,a)$
acting as $f\mapsto -f''$ on the functions $f\in H^2(0,a)$ satisfying the boundary conditions
$f'(0)+\beta f(0)=f'(a)-\gamma f(a)=0$,
and let $E(a,\beta,\gamma)$ be its lowest eigenvalue.
Let $\beta>2\gamma$ and $\beta a>1$, then $
\beta^2< -E(a,\beta,\gamma)<\beta^2+123\beta^2 e^{-2\beta a}$.
\end{lemma}

\begin{proof}
Let $k>0$. Clearly, $E=-k^2$ is an eigenvalue of $\Lambda_{a,\beta,\gamma}$
iff one can find  $(C_1,C_2)\in\CC^2\setminus\big\{(0,0)\big\}$ such that the function
$f:x\mapsto C_1e^{kx}+C_2 e^{-kx}$ belongs to the domain of  $\Lambda_{a,\beta,\gamma}$.
The boundary conditions give
\begin{align*}
0=f'(0)+\beta f(0)&=(\beta+k)C_1 + (\beta-k)C_2,\\
0=f'(a)-\gamma f(a)&=(k-\gamma)e^{ka} C_1 - (k+\gamma)e^{-ka}C_2,
\end{align*}
and one has a non-zero solution iff the determinant of this system vanishes, i.e.
iff $k$ satisfies the equation $(k+\beta)(k+\gamma)e^{-ka}=(k-\beta)(k-\gamma)e^{ka}$.
Let us look for solutions $k\in (\beta,+\infty)$. One may rewrite the preceding equation
as
\begin{equation}
    \label{eq-gh}
g(k)=h(k), \quad g(k)=\dfrac{k+\beta}{k-\beta}, \quad h(k)=\dfrac{k-\gamma}{k+\gamma} e^{2ka}.
\end{equation}
Both functions $g$ and $h$ are continuous. It is readily seen that the function $g$ is strictly decreasing
on $(\beta,+\infty)$ with $g(\beta+)=+\infty$ and $g(+\infty)$=1. On the other hand, for $\beta>2\gamma$
the function $h$ is strictly increasing in $(\beta,+\infty)$ 
being the product of two strictly increasing positive
functions, and we have
$h(\beta+)=e^{2\beta a}(\beta-\gamma)/(\beta+\gamma) <+\infty$
and $h(+\infty)=+\infty$. These properties of $g$ and $h$ show that
there exists a unique solution $k=k(a,\beta,\gamma)\in (\beta,+\infty)$ of \eqref{eq-gh}
and that $E(a,\beta,\gamma)=-k(a,\beta,\gamma)^2$.

To obtain the required estimate we use again the monotonicity of $h$ on $(\beta,+\infty)$
and the inequality $\beta>2\gamma$. We have
\[
\dfrac{k+\beta}{k-\beta}=g(k)=h(k)> h(\beta+)=\dfrac{\beta-\gamma}{\beta+\gamma} e^{2\beta a}\ge \dfrac{e^{2\beta a}}{3},
\]
which gives $(1-3e^{-2\beta a})k<(1+3e^{-2\beta a})\beta$. The assumption $\beta a>1$ gives the inequality
$3e^{-2\beta a}<1/2$, and we arrive at
\[
k<\dfrac{1+3e^{-2\beta a}}{1-3e^{-2\beta a}}\,\beta<
(1+3e^{-2\beta a})(1+15e^{-2\beta a})\beta<(1+41e^{-2\beta a})\beta
\]
and $k^2<(1+41e^{-2\beta a})^2\beta^2< (1+123 e^{-2\beta a})\beta^2$.
Together with the inclusion $k \in(\beta,+\infty)$ this gives the result.
\end{proof}

\begin{lemma}\label{lem3}
For $a,\beta>0$, let $\Pi_{a,\beta}$
denote the self-adjoint operator in $L^2(0,a)$
acting as $f\mapsto -f''$ on the functions
$f\in H^2(0,a)$ satisfying the boundary conditions $f'(0)+\beta f(0)=f(a)=0$,
and let $E(a,\beta)$ be its lowest eigenvalue. Assume that $\beta a>4/3$, then
$\beta^2-4\beta^2 e^{-\beta a}<-E(a,\beta)<\beta^2$.
\end{lemma}

\begin{proof}
Let $k>0$. Proceeding as in the proof of lemma~\ref{lem2} we see that $E=-k^2$ is an eigenvalue of $\Pi_{a,\beta}$
iff $k$ satisfies the equation $(\beta+k)e^{-ka}=(\beta-k)e^{ka}$. As the left-hand side is strictly positive,
the right-hand side must be positive too, which means that all solutions $k$ belong to $(0,\beta)$.
Let us rewrite the equation in the form $g(k)=0$ with $g(k):=\log(\beta+k)-\log(\beta-k) -2ka$.
One has $g(0)=0$, the function $g$ is strictly decreasing in $(0,k_0)$ and strictly increasing in $(k_0,\beta)$,
with $k_0:=\sqrt{\beta^2-\beta/a}$. Moreover, $g(\beta-)=+\infty$. Therefore,
the equation $g(k)=0$ has a unique solution in $(k_0,\beta)$. It follows from the assumption
$\beta a>4/3$ that $k_0>\beta/2$, and we can represent $k=\beta-s$
with some $s\in(0,\beta/2)$. Using again the condition $g(k)=0$
we arrive at the inequality $\log s=\log(2\beta-s)-2\beta a +2sa< \log (2\beta)-\beta a$,
which gives $s< 2\beta e^{-\beta a}$ and
$k=\beta-s>\beta(1-2e^{-\beta a})$. Finally,
$-E(a,\beta)=k^2>\beta^2(1-2e^{-\beta a})^2>\beta^2(1-4e^{-\beta a})$.
Together with the first inequality $k<\beta$ this gives the estimate desired.
\end{proof}

Let us complete the proof of lemma~\ref{prop1}.
Denote $a_1:=\min\big\{a_0,(10KC)^{-1}\big\}$ and pick any $a\in (0,a_1)$, and let
$\beta>3K+1+4/(3a)$.
Applying lemma~\ref{lem2} to each of the operators $L^{-,j}_{\beta,a}$
and lemma~\ref{lem3}  to each of the operators $L^{+,j}_{\beta,a}$
we arrive at the estimates
\begin{align*}
E^{-,j}(\beta,a)&>-\Big(\beta+\dfrac{\kappa^+_{M,j}}{2}\Big)^2-123 \Big(\beta+\dfrac{\kappa^+_{M,j}}{2}\Big)^2
\exp \bigg[-2a\Big(\beta+\dfrac{\kappa^+_{M,j}}{2}\Big)\bigg],\\
E^{+,j}(\beta,a)&<-\Big(\beta+\dfrac{\kappa^-_{M,j}}{2}\Big)^2 +
4\Big(\beta+\dfrac{\kappa^-_{M,j}}{2}\Big)^2
\exp \bigg[-a\Big(\beta+\dfrac{\kappa^-_{M,j}}{2}\Big)\bigg].
\end{align*}
To simplify the form of the remainders we choose $\beta_a>0$ sufficiently large such that
for $\beta>\beta_a$ we have
\[
\Big(\beta+\dfrac{K}{2}\Big)^2
\exp \bigg[-2a\Big(\beta-\dfrac{K}{2}\Big)\bigg]
+
4\Big(\beta+\dfrac{K}{2}\Big)^2
\exp \bigg[-a\Big(\beta-\dfrac{K}{2}\Big)\bigg]
\le \dfrac{1}{\beta},
\] 
then for $\beta>\beta_a+3K+1+4/(3a)$ and all $j=1,\dots,M$ we have
\[
E^{-,j}(\beta,a)>-\beta^2 -\kappa_{M,j}^+ \beta -\dfrac{K^2}{4} - \dfrac{1}{\beta},\quad
E^{+,j}(\beta,a)<-\beta^2 -\kappa^-_{M,j} \beta  + \dfrac{1}{\beta}.
\]
Using the inequality $\kappa^+_{M,j}\le\kappa_\mx$ we obtain
\begin{equation}
       \label{eq-est1}
\min_j E^{-,j}(\beta,a)> -\beta^2 -\kappa_\mx \beta -\dfrac{K^2}{4} - \dfrac{1}{\beta}.
\end{equation}
On the other hand, let $l\in\{1,\dots, M\}$ be such that $\kappa^+_{M,l}=\kappa_\mx$.
This means that there exists $s\in \overline{I^l_M}$ such that $\kappa(s)=\kappa_\mx$.
Using the Taylor expansion near $s$ we obtain
\begin{equation}
    \label{eq-kappa}
\kappa^-_{M,l}\ge \kappa^+_{M,l} - K\delta=\kappa_\mx - K\delta\equiv
\kappa_\mx - \dfrac{K\ell}{M}.
\end{equation}
In the previous considerations the number $M$ was arbitrary, and now we pick 
$M\in\big[\,\beta^\frac{1}{3},2\beta^\frac{1}{3}\big]\mathop{\cap} \NN$, then
\begin{multline}
    \label{eq-est2}
\min_j E^{+,j}(\beta,a)\le E^{+,l}(\beta,a)<-\beta^2 -\kappa^-_{M,l} \beta  + \dfrac{1}{\beta}\\
=-\beta^2 -\kappa_\mx \beta +\dfrac{K \ell}{M}\beta+ \dfrac{1}{\beta}\le -\beta^2 -\kappa_\mx \beta +K \ell \beta^\frac{2}{3} + \dfrac{1}{\beta}.
\end{multline}
Substituting the estimates \eqref{eq-est1} and \eqref{eq-est2} into \eqref{eq-eig5} we arrive at
\begin{align*}
E^+_M(\beta,a)&\le -\beta^2 -\kappa_\mx \beta +K \ell \beta^\frac{2}{3} + \dfrac{1}{\beta}
+\dfrac{4\pi^2 M^2}{\ell^2}+v \\
&= -\beta^2 -\kappa_\mx \beta +\Big(K\ell + \dfrac{16\pi^2}{\ell^2}\Big)\beta^\frac{2}{3} + v +\dfrac{1}{\beta},\\
E^-_M(\beta,a)&\ge 
-\beta^2 -\kappa_\mx \beta -\dfrac{K^2}{4} - v-4aKC-\dfrac{1}{\beta},
\end{align*}
and the assertion of lemma~\ref{prop1} follows from 
the two-side estimates \eqref{eq-eig4} .

\section{Proof of Theorem~\ref{thm1}}

We continue using the notation introduced just before theorem~\ref{thm1}.
For $a>0$ consider the maps
\[
\Phi_{k,a}:(0,\ell_k)\times\RR\to \RR^2,\quad
\Phi_{k,a}(s,u)=\begin{pmatrix}
\Gamma_{k,1}(s)-u\Gamma'_{k,2}(s)\\
\Gamma_{k,2}(s)+u\Gamma'_{k,1}(s)
\end{pmatrix}, \quad k=1,\dots, n.
\]
As in section~\ref{sec2}, we can find $a_0>0$ such that for any $a\in(0,a_0)$
these maps are diffeomorphic between $\square_{k,a}:=(0,\ell_k)\times(0,a)$ and $\Omega_{k,a}:=\Phi_{k,a}(\square_{k,a})$,
that $\Omega_{k,a}\subset \Omega$, and that $\Omega_{j,a}\cap \Omega_{k,a}=\emptyset$ for $j\ne k$.
Note that the last property follows from the fact that the opening angles of the boundary corners (if any)
are reflex. In addition, we set $\Omega_{0,a}:=\Omega \setminus \Big(\bigcup_{k=1}^n \overline{\Omega_{k,a}}\Big)$.
Denote $\Tilde H^1_0(\Omega_{k,a}):=\big\{
f\in H^1(\Omega_{k,a}): \, f\uhr_{\partial \Omega_{k,a}\setminus{\Bar\Sigma_k}}=0
\big\}$, $k=1,\dots,n$, and introduce two new sesquilinear forms $h^{N/D,a}_\beta$ in $L^2(\Omega)$,
both defined by the same expression as $h_\beta$ on the domains
\[
\dom h^{N,a}_\beta=\bigoplus_{k=0}^n H^1(\Omega_{k,a}),
\quad
\dom h^{D,a}_\beta=H^1_0(\Omega_{0,a})\mathop{\cup}\bigg(\bigoplus_{k=1}^n \Tilde H^1_0(\Omega_{k,a})\bigg),
\]
and define
\[
E_{N/D}(\beta,a):=\inf_{0\ne f\in \dom h^{N/D,a}_\beta}\dfrac{h^{N/D,a}_\beta(f,f)}{\|f\|^2_{L^2(\Omega)}}.
\]
Due to the inclusions $\dom h^{D,a}_\beta\subset\dom h_\beta\subset \dom h^{N,a}_\beta$ we have the inequalities
\begin{equation}
      \label{eq-ts}
E_{N}(\beta,a)\le E(\beta)\le E_{D}(\beta,a).
\end{equation}
 Furthermore, due to the fact that the parts $\Omega_{k,a}$ are disjoint
 and that the set $\Sigma\mathop{\cap}\partial\Omega_{0,a}$ is finite (this is exactly
 the set of the corners),
we have the equality
$E_{N/D}(\beta,a)=\min_{k\in\{0,\dots,n\}} E_{k,N/D}(\beta,a)$,
with
\begin{align*}
E_{0,N}(\beta,a)&:=\inf_{0\ne f\in H^1(\Omega_{0,a})}\dfrac{\|\nabla f\|^2_{L^2(\Omega_{0,a})}}{\|f\|^2_{L^2(\Omega_{0,a})}},\\
E_{k,N}(\beta,a)&:=
\inf_{0\ne f\in H^1(\Omega_{k,a})}\dfrac{\|\nabla f\|^2_{L^2(\Omega_{k,a})}-\beta\|f\|^2_{L^2(\Sigma_k)}}{\|f\|^2_{L^2(\Omega_{k,a})}},
\quad k=1,\dots,n,\\
E_{0,D}(\beta,a)&=\inf_{0\ne f\in H^1_0(\Omega_{0,a})}\dfrac{\|\nabla f\|^2_{L^2(\Omega_{0,a})}}{\|f\|^2_{L^2(\Omega_{0,a})}},\\
E_{k,D}(\beta,a)&:=
\inf_{0\ne f\in \Tilde H^1_0(\Omega_{k,a})}\dfrac{\|\nabla f\|^2_{L^2(\Omega_{k,a})}-\beta\|f\|^2_{L^2(\Sigma_k)}}{\|f\|^2_{L^2(\Omega_{k,a})}},
\quad k=1,\dots,n.
\end{align*}
We have clearly $E_{0,N/D}(\beta,a)\ge 0$.
Furthermore, in virtue of lemma~\ref{prop1} we can find $a>0$ such that
for each $k\in\{1,\dots,n\}$ for $\beta\to+\infty$ we have
\[
E_{k,N/D}(\beta,a)=-\beta^2-\gamma_{k,\mx}\beta+O\big(\beta^\frac{2}{3}\big), \quad \gamma_{k,\mx}:=\max_{s\in[0,\ell_k]}\gamma_k(s),
\]
which gives $E_{N/D}(\beta,a)=-\beta^2-\gamma_\mx\beta+O\big(\beta^\frac{2}{3}\big)$, and the assertion
of theorem~\ref{thm1} follows from the two-side estimate \eqref{eq-ts}.

\begin{remark}
A more detailed asymptotic analysis is beyond the scope of the present note, but we mention one case
in which the remainder estimate can be slightly improved with minimal efforts. Namely,
assume that one of the following conditions is satisfied:
\begin{itemize}
\item the boundary $\Sigma$ is of class $C^4$ (i.e. there are no corners),
\item the curvature does not attain its maximal value $\gamma_\mx$ at the corners,
\end{itemize}
then
\begin{equation}
   \label{eq-impr}
E(\beta)=-\beta^2-\gamma_\mx\beta +O\big(\sqrt\beta\,\big) \text{ as } \beta\to+\infty.
\end{equation}
Indeed, let us pick any $k\in\{1,\dots, n\}$ such that $\gamma_{k,\mx}=\gamma_\mx$
and revise the proof of lemma~\ref{prop1} with $\Gamma:=\Gamma_k$, $\kappa:=\gamma_k$ and $\ell:=\ell_k$.
For any $s\in[0,\ell]$ with $\kappa(s)=\kappa_\mx$ we have then $\kappa'(s)=0$,
and we may replace the inequality \eqref{eq-kappa} with
\[
\kappa^-_{M,l}\ge \kappa^+_{M,l} - K\delta^2=\kappa_\mx - K\delta^2\equiv
\kappa_\mx - \frac{K\ell^2}{M^2},
\]
and by choosing $M\in\big[\sqrt[4]{\mathstrut \beta}, \, 2\sqrt[4]{\mathstrut\beta}\,\big]\mathop{\cap}\NN$ we arrive
at the estimate $E_{N/D}(\beta,a)=-\beta^2-\kappa_\mx \beta +O(\sqrt{\mathstrut\beta}\,)$ as $\beta\to+\infty$, which in turn
gives the asymptotics \eqref{eq-impr}.

\end{remark}

\section{Acknowledgments}
The research was partially supported by ANR NOSEVOL and GDR Dynamique quantique.

\end{document}